\documentclass[a4paper,10pt]{amsproc}

\usepackage{amsfonts, amsmath, amssymb, graphicx, mathrsfs}
\usepackage[all]{xy}

\newdir{ >}{!/6pt/@{ }*:(1,0.2)@_{>}*:(1,-0.2)@^{>}}

\theoremstyle{plain}

\newtheorem{theorem}{Theorem}[section]

\theoremstyle{definition}

\newtheorem{remark}[theorem]{Remark}
\newtheorem{counter example}[theorem]{Counter Example}

\newtheorem{corollary}[theorem]{Corollary}
\newtheorem{example}[theorem]{Example}
\numberwithin{equation}{section}

\DeclareMathOperator{\cl}{cl}

\begin{document}

\title[Essential Ideals]{Some new results on functions in $C(X)$ having their support on ideals of closed sets}

\author[S. Bag]{Sagarmoy Bag}
\address{Department of Pure Mathematics, University of Calcutta, 35, Ballygunge Circular Road, Kolkata 700019, West Bengal, India}
\email{sagarmoy.bag01@gmail.com}
\author[S.K.Acharyya]{Sudip Kumar Acharyya}
\address{Department of Pure Mathematics, University of Calcutta, 35, Ballygunge Circular Road, Kolkata 700019, West Bengal, India}
\email{sdpacharyya@gmail.com}
\thanks{The first author thanks the NBHM, Mumbai-400 001, India, for financial support}

\author[P.Rooj]{Pritam Rooj}
\address{NIIT University, Neemrana, Rajasthan, PIN 301705, India }
\email{pritam.rooj@niituniversity.in}

\author[G. Bhunia]{Goutam Bhunia}
\address{Department of Pure Mathematics, University of Calcutta, 35, Ballygunge Circular Road, Kolkata 700019, West Bengal, India}
\email{bhunia.goutam72@gmail.com}
\thanks{The fourth author thanks the NBHM, Mumbai-400 001, India, for financial support}


\keywords{Compact support, pseudocompact space, intermediate ring,                                                                                                                                                                                                                                                                                                                                                                                                                                                                                                                                                                                                                                                                                                                                                                                                                                                                                                                                                                                                                                                                                                                                                                                                                                                                                                                                                                                                                                                                                                                                                                                                                                                                                                                                                                                                                                                                                                                                                                                                                                                                                                                                                                                                                                                                                                                                                                                                                                                                                                                                                                                                                                                                                                                                                                                                                                                                                                                                                                                                                                                                                                                                                                                                                                                                                                                                                                                                                                                                                                                                                                                                                                                                                                                                                                                                                                                                                                                                                                                                                                                                                                                                                                                                                                                                                                                                                                                                                                                                                                                                    pseudocompact support, essential ideal, $z^\circ$-ideal, socle, $C$-type ring.}
\thanks {}

\maketitle
\begin{abstract}
	For any ideal $\mathcal{P}$ of closed sets in $X$, let $C_\mathcal{P}(X)$ be the family of those functions in $C(X)$ whose support lie on $\mathcal{P}$. Further let $C^\mathcal{P}_\infty(X)$ contain precisely those functions $f$ in $C(X)$ for which for each $\epsilon >0, \{x\in X: \lvert f(x)\rvert\geq \epsilon\}$ is a member of $\mathcal{P}$. Let $\upsilon_{C_{\mathcal{P}}}X$ stand for the set of all those points $p$ in $\beta X$ at which the  stone extension $f^*$ for each $f$ in $C_\mathcal{P}(X)$ is real valued. We show that each realcompact space lying between $X$ and $\beta X$ is of the form $\upsilon_{C_\mathcal{P}}X$ if and only if $X$ is pseudocompact. We find out conditions under which an arbitrary product of spaces of the form locally-$\mathcal{P}/$ almost locally-$\mathcal{P}$, becomes a space of the same form. We further show that $C_\mathcal{P}(X)$ is a free ideal ( essential ideal ) of $C(X)$ if and only if $C^\mathcal{P}_\infty(X)$ is a free ideal ( respectively essential ideal ) of $C^*(X)+C^\mathcal{P}_\infty(X)$ when and only when $X$ is locally-$\mathcal{P}$ ( almost locally-$\mathcal{P}$). We address the problem, when does $C_\mathcal{P}(X)/C^\mathcal{P}_{\infty}(X)$ become identical to the socle of the ring $C(X)$. Finally we observe that the ideals of the form $C_\mathcal{P}(X)$ of $C(X)$ are no other than the $z^\circ$-ideals of $C(X)$. 
\end{abstract}	

\section{Introduction}
In what follows $C(X)$ stands for the ring of all real valued continuous functions on a completely regular Hausdorff topological space $X$. $C^*(X)$, as usual denotes the subring of $C(X)$ containing those functions, which are bounded over $X$. Let $\mathcal{P}$ be a family of closed subsets of $X$ satisfying the following two conditions: 
\begin{enumerate}
	\item If $A, B\in \mathcal{P}$ then $A\cup B\in\mathcal{P}$.
	\item If $A\in\mathcal{P}$ and $B\subseteq A$ with $B$ closed in $X$ then $B\in \mathcal{P}$
\end{enumerate}
  i.e. $\mathcal{P}$ is an ideal of closed sets in $X$. For any $f\in C(X)$,  $Z(f)=\{x\in X:f(x)=0\}$ stands for the zero set of $f$. Let $\Omega (X)$ denote the set of all ideals of closed sets in $X$. For each $\mathcal{P}\in\Omega (X)$ set $C_{\mathcal{P}}(X)=\{f\in C(X):\cl_X(X-Z(f))\in \mathcal{P}\}$ and $C^{\mathcal{P}}_{\infty}(X)=\{ f\in C(X): \{x\in X:\lvert f(x)\rvert\geq \epsilon\}\in \mathcal{P}\text{ for each }\epsilon >0\}$. If $\mathcal{P}$ is chosen to be the family of compact subsets of $X$, then $C_{\mathcal{P}}(X)$ and $C^{\mathcal{P}}_{\infty}(X)$ coincide with the ring $C_K(X)$ of real valued continuous functions on $X$ with compact support and the ring $C_{\infty}(X)$ of real valued continuous functions which vanish at infinity respectively. These two rings $C_{\mathcal{P}}(X)$ and $C_{\mathcal{P}}^\infty(X)$ have been introduced first time in 2010 in \cite{AG}. However the same pair of rings have been reintroduced by following some dual techniques in 2015 in \cite{T}.
  
  As usual $\beta X$ stands for the Stone-$\check{C}$ech compactification of $X$. For any $f\in C(X), f^*:\beta X\mapsto \mathbb{R}\cup \{\infty\}$ is its unique continuous extension over $\beta X$ into the one-point compactification of $\mathbb{R}$. For $\mathcal{P}\in \Omega (X)$, let $\upsilon_{C_{\mathcal{P}}} X=\{ p\in\beta X:f^*(p)\in\mathbb{R}\text{ for all } f\in C_{\mathcal{P}}(X)\}$. It is easy to see that $\upsilon_{C_{\mathcal{P}}}X$ is a real compact space lying between $X$ and $\beta X$. The first main theorem of this article states that each real compact space between $X$ and $\beta X$ could be achieved as a space of the form $\upsilon_{C_{\mathcal{P}}}X$ for some $\mathcal{P}\in \Omega (X)$ if and only if $X$ is pseudocompact (Theorem \ref{3}).
  
  For any $\mathcal{P}\in\Omega (X)$, we call $X$ locally-$\mathcal{P}$ if each point $x$ of $X$ has an open neighbourhood whose closure in $X$ belongs to $\mathcal{P}$. More generally $X$ is called almost locally-$\mathcal{P}$ if each nonempty open set $U$ in $X$ contains a nonempty open set $V$ such that $\cl_X V$ is a member of $\mathcal{P}$. If $\mathcal{P}$ is the ideal of all compact sets in $X$, then locally-$\mathcal{P}$ spaces and almost locally-$\mathcal{P}$ spaces bog down to locally compact spaces and almost locally compact spaces respectively. The later class of spaces is introduced in \cite{A}. We have introduced the notion product of ideals of closed sets in the following manner: If $\{X_{\alpha} :\alpha\in\Lambda \}$ is an indexed family of topological spaces with $\mathcal{P}_{\alpha}\in\Omega(X_{\alpha})$ for each $\alpha\in\Lambda$, then $\prod_{\alpha\in\Lambda}\mathcal{P}_{\alpha}$ is defined to be the smallest member of $\Omega (\prod_{\alpha\in\Lambda}X_{\alpha})$ containing sets of the form $\prod_{\alpha\in\Lambda}A_{\alpha}$, where $A_{\alpha}$ belongs to $\mathcal{P}_{\alpha}$ for each $\alpha$ in $\Lambda$. We have shown that the product space $\prod_{\alpha\in \Lambda}X_{\alpha}$ is locally-$\prod_{\alpha\in\Lambda}\mathcal{P}_{\alpha}$ if and only if $X_{\alpha}$ is locally-$\mathcal{P}_{\alpha}$ for each $\alpha\in\Lambda$ and $X_{\alpha}\in\mathcal{P}_{\alpha}$ for all but finitely many $\alpha$'s in $\Lambda$. We have further checked that $\prod_{\alpha\in\Lambda}X_{\alpha}$ is almost locally-$\prod_{\alpha\in\Lambda}\mathcal{P}_{\alpha}$ when and only when $X_{\alpha}$ is almost locally-$\mathcal{P}_{\alpha}$ for each $\alpha$'s in $\Lambda$ and $X_\alpha\in\mathcal{P}_\alpha$ for all but finitely many $\alpha$'s in$\lambda$. For the particular choice of $\mathcal{P}_{\alpha}$ as the ideal of all compact sets in $X_{\alpha}$ for each $\alpha$ in $\Lambda$, the first result reads: $\prod_{\alpha\in \Lambda}X_{\alpha}$ is locally compact if and only if each $X_{\alpha}$ is locally compact and all but finitely many $X_{\alpha}$'s are compact, a well known result in general topology. For the same choice of $\mathcal{P}_{\alpha}$'s, the second result says that $\prod_{\alpha\in\Lambda}X_{\alpha}$ is almost locally compact if and only if $X_{\alpha}$ is almost locally compact for each $\alpha$ in $\Lambda$ and $X_{\alpha}$'s are compact for all but finitely many $\alpha$'s in $\Lambda$. This later fact is established in \cite{A}. Furthermore on choosing $\mathcal{P}_{\alpha}$ as ideal of closed bounded sets in $X_{\alpha}$we have established analogous results involving locally pseudocompact spaces and almost locally pseudocompact spaces. It is easy to check that $C_{\mathcal{P}}(X)\subseteq C^\mathcal{P}_{\infty}(X)$ with $C_{\mathcal{P}}(X)$ an ideal of the ring $C(X)$ and $C^{\mathcal{P}}_{\infty}(X)$ a subring of $C(X)$. For some specific choice of $\mathcal{P}$, say $\mathcal{P}\equiv$ the ideal of all compact sets in $X\equiv K$ say, $C^{\mathcal{P}}_{\infty}(X)=C_{\infty}(X)\subseteq C^*(X)$. In general however $C^{\mathcal{P}}_{\infty}(X)$ may not be contained in $C^*(X)$, this is easily seen on taking $X=\mathbb{R}$ and $\mathcal{P}$, the ideal of all closed Lindel\"{o}f subsets of $X$. But if $f\in C^{\mathcal{P}}_{\infty}(X)$ and $g\in C^*(X)$, then it follows that $fg\in C^{\mathcal{P}}_{\infty}(X)$. Hence $C^{\mathcal{P}}_{\infty}(X)$ is an ideal of the ring $C^{\mathcal{P}}_{\infty}(X)+C^*(X)\equiv \{f+g: f\in C^{\mathcal{P}}_{\infty}(X), g\in C^*(X)\}$. We have shown that $C_{\mathcal{P}}(X)$ is a free ideal of $C(X)$ if and only if $C^{\mathcal{P}}_{\infty}(X)$ is a free ideal of $C^{\mathcal{P}}_{\infty}(X)+C^*(X)$ if and only if $X$ is locally-$\mathcal{P}$. We recall that an ideal $I$ in $C(X)$ is free if $\cap_{f\in I}Z(f)=\phi$. We have established that for $\mathcal{P}\in \Omega (X)$, $C_{\mathcal{P}}(X)$ is an essential ideal of $C(X)$ when and only when $C^{\mathcal{P}}_{\infty}(X)$ is an essential ideal of $C^{\mathcal{P}}_{\infty}(X)+C^*(X)$ and this is precisely the case if and only if $X$ is almost locally-$\mathcal{P}$. With $\mathcal{P}\equiv K$, this reads: $C_K(X)$ is an essential ideal of $C(X)$ if and only if $C_{\infty}(X)$ is an essential ideal of $C^*(X)$ when and only when $X$ is almost locally compact. This last result was established in \cite{A}. Incidentally a nonzero ideal $I$ in a commutative ring $R$ (with or without identity) is called an essential ideal if it intersects each nonzero ideal in $R$ nontrivially, this means that for each $a\neq 0$ in $R$, there exists $b\neq 0$ in $R$ such that $ab\neq 0$ and $ab\in I$. We have recorded a second special case of our general result involving $C_{\mathcal{P}}(X)$ and $C^{\mathcal{P}}_{\infty}(X)$, mentioned above in the following manner: the ring $C_{\psi}(X)$ of all functions in $C(X)$ with pseudocompact support is an essential ideal of $C(X)$ if and only if the ring $C^{\psi}_{\infty}(X)\equiv \{f\in C(X):\text{ for each } \epsilon>0 \text{ in }\mathbb{R}, \cl_X\{x\in X:\lvert f(x)\rvert > \epsilon\}\text{ is pseudocompact}\}$ is an essential ideal of $C^*(X)$ when and only when $X$ is almost locally pseudocompact meaning that each nonempty open set $U$ in $X$ contains nonempty open set $V$ such that $\cl_XV$ is pseudocompact. The ring $C^{\psi}_{\infty}(X)$ may be called the pseudocompact analogue of the ring $C_{\infty}(X)$. As a follow up of the results of this type, we have addressed the problem: when does $C_{\mathcal{P}}(X)$ (respectively $C^{\mathcal{P}}_{\infty}(X)$) become identical to the intersection of all essential ideals of $C(X)$ also known as the socle of $C(X)$. We conclude this article after establishing that ideals of the ring $C(X)$ of the form $C_{\mathcal{P}}(X)$ with $\mathcal{P}\in \Omega (X)$ are precisely the $z^\circ$-ideals of $C(X)$, investigated amongst others by Azarpanah, Karamzadeh, Aliabad and Karavan in a series of articles \cite{AKA},\cite{AKAL},\cite{AKAR},\cite{AA}. On realizing $z^\circ$-ideals of $C(X)$ as ideals of the type $C_{\mathcal{P}}(X)$, we have given an alternative short proof of the fact that each ideal of $C(X)$ consisting only of zero divisors extends to a $z^\circ$-ideal of $C(X)$. For more information about the rings $C_\mathcal{P}(X)$ and $C^\mathcal{P}_\infty(X)$, we refer to the articles \cite{AG}, \cite{AS}. For undefined terms see \cite{G}.
  
\section{A Characterization of pseudocompact spaces $X$ in terms of ideals of the form $C_{\mathcal{P}}(X)$}
  For a subset $A$ of $C(X)$ and a subset $T$ of $\beta X$, we will use the notation $\upsilon_AX=\{p\in\beta X: f^*(p)\in\mathbb{R}\text{ for each } f\in A\}$ and $C_T=\{ f\in C(X):f^*(p)\in\mathbb{R}\text{ for each }p\in T\}$. Clearly $\upsilon_CX=\upsilon X$, the Hewitt real compactification of $X$ and $\upsilon_{C^*}X=\beta X$. A ring is said to be of $C$-type if it is isomorphic to a ring of the form $C(X)$ for some space $X$. By an intermediate ring we mean a ring that lies between $C^*(X)$ and $C(X)$. Let $\Sigma (X)$ stand for the set of all intermediate rings. For $A, B\in\Sigma(X)$ we write $A\sim B$ if and only if $\upsilon_AX=\upsilon_BX$. $'\sim'$ is an equivalence relation on $\Sigma(X)$ causing a partition of this set into disjoint equivalence classes. Each equivalence class has a largest member, the largest member of the equivalence class of $A(X)$ is given by $\{g\mid_X:g\in C(\upsilon_AX)\}$, which contains precisely those functions in $C(X)$ which have continuous extension over the $A$-compactification $\upsilon_AX$of $X$. Incidentally $C^*(X)$ is the lone member of its own equivalence class. It is easy to see that for any subset $T$ of $\beta X$, $C_T$ is an intermediate ring. The following fact characterizing $C$-type rings amongst members of $\Sigma(X)$ is established in \cite{DA}:

\begin{theorem}\label{1}
For a ring $A(X)\in \Omega(X)$, the following three statements are equivalent:
\begin{enumerate}
\item $A(X)$ is a $C$-type ring.
\item $A(X)$ is the largest member of its own equivalence class modulo the relation $'\sim'$ on $\Sigma(X)$.
\item There exists a subset $T$ of $\beta X$ such that such that $A(X)=C_T$.
\end{enumerate}
\end{theorem}  
 We also write down the following proposition giving some natural examples of $C$-type intermediate rings as recorded in \cite{DO}.
 
\begin{theorem}\label{2}
For any ideal $I$ of $C(X)$, $C^*(X)+I=\{f+g:f\in C^*(X)\text{and } g\in I\}$ is a $C$-type intermediate ring.
\end{theorem}

We are now ready to establish the first main result of the present paper.

\begin{theorem}\label{3}
 Each real compact space lying between $X$ and $\beta X$ can be achieved as a space of the form $\upsilon_{C_{\mathcal{P}}}X$ for some $\mathcal{P}\in\Omega (X)$ if and only if $X$ is pseudocompact.
\end{theorem}

\begin{proof}
 If $X$ is pseudocompact then $\beta X=\upsilon X$, so that $\beta X$ is the only real compact space lying between $X$ and $\beta X$ and we can write $\beta X=\upsilon_{C_{\mathcal{P}}}X$, where we can take $\mathcal{P}\equiv \mathcal{K}$. To prove the other part of this theorem assume that $X$ is not pseudocompact. We shall construct a real compact space lying between $X$ and $\beta X$ which is not of the form $\upsilon_{C_{\mathcal{P}}}X$ for a $\mathcal{P}\in\Omega(X)$. As $X$ is not pseudocompact, there exists a copy of $\mathbb{N}=\{1, 2, 3, .....\}$, $C$-embedded in $X$. We make a partition of $\mathbb{N}$ as follows: $\mathbb{N}=\cup_{k\in\omega_\circ}\mathbb{N}_{k}$, where each $\mathbb{N}_k$ is an infinite set and for $k\neq j$, $\mathbb{N}_k\cap\mathbb{N}_j=\phi$. Since $X$ is $C$-embedded in $\upsilon X$ it follows that for each $k\in \omega_\circ$, $\mathbb{N}_K$ is $C$-embedded in $\upsilon X$. As every $C$-embedded countable subset of a Tychonoff space $Y$ is a closed subset of $Y$ [see Ex. 3B3, Chapter 3, Gillman-Jerison text \cite{G}],it turns out that each $\mathbb{N}_k$ is a closed subset of $\upsilon X$. Since $\mathbb{N}_k$ is noncompact and therefore not closed in $\beta X$, it follows that $\cl_{\beta X}\mathbb{N}_k\setminus\mathbb{N}_k\subseteq \beta X\setminus\upsilon X$. Since disjoint zero sets in a Tychonoff space $Y$ have their closures in $\beta Y$ disjoint, this further implies that $\{ \cl_{\beta X}\mathbb{N}_k\setminus\mathbb{N}_k:k\in\omega_\circ\}$, constitutes a countably infinite family of pairwise disjoint nonempty subsets of $\beta X\setminus \upsilon X$, indeed each such subset is a copy of $\beta \mathbb{N}\setminus\mathbb{N}$. For each $k\in\omega_\circ$, we select a point $p_k\in \cl_{\beta X}\mathbb{N}_k\setminus\mathbb{N}_k$. Let $T=\{p_1,p_2,...\}$ and recall that $C_T=\{f\in C(X):f^*(p_k)\in\mathbb{R}\text{ for each } k\in\omega_\circ\}$. It follows from Theorem  \ref{1} that $C_T$ is a $C$-type intermediate ring. We now assert that there does not exist any ideal $I$ of $C(X)$ such that $C_T=C^*(X)+I$. We argue by contradiction and suppose $C_T=C^*(X)+I$ for some ideal $I$ of $C(X)$. Since $\mathbb{N}$ is $C$-embedded in $X$, there exists an $f\in C(X)$ such that $f(\mathbb{N}_k)=k$ for each $k\in\omega_\circ$. Consequently $f^*(\cl_{\beta X}\mathbb{N}_k)=k$ for each $k\in\omega_\circ$ in particular  $f^*(p_k)=k$. Thus $f^*$ takes real values on the whole of $T$, hence $f\in C_T$. So we can write $f=f_1+f_2$, where $f_1\in C^*(X)$ and $f_2\in I$. Now  for a sufficiently large value of $k$, we can write $f^*(p_k)>1$. But as $p_k$ is a hyperreal point of $\beta X$ i.e. $p_k\in\beta X\setminus \upsilon X$, we can find out an $u\in C(X)$ with $Z(u)=\phi$ such that $u^*(p_k)=\infty$. It follows that $(uf_2)^*(p_k)=\infty$. On the other hand, the fact that $I$ is an ideal of $C(X)$ and $f_2\in I$ implies that $uf_2\in I$. Therefore $uf_2\in C_T$ and hence $(uf_2)^*(p_k)\in\mathbb{R}$. We arrive at a contradiction. Thus it is settled that $C_T\neq C^*(X)+I$ for any ideal $I$ of $C(X)$. To complete this theorem we shall prove that the realcompact space $\upsilon_{C_T}X\equiv\{p\in\beta X:f^*(p)\in\mathbb{R}\text{ for each } f\in C_T\}$ cannot be expressed in the form $\upsilon_{C_{\mathcal{P}}}X$ for any $\mathcal{P}\in \Omega(X)$. This time also we argue by contradiction and assume that there is an ideal $\mathcal{P}$ of closed sets in $X$ for which $\upsilon_{C_T}X=\upsilon_{C_{\mathcal{P}}}X$. It follows that $\upsilon_{C_T}X=\upsilon_{C_{\mathcal{P}}(X)+C^*(X)}X$, this means that $C_T$ and $C_{\mathcal{P}}(X)+C^*(X)$ belong to the same equivalence class of $\Sigma (X)$modulo the relation $'\sim'$, introduced earlier. But $C_T$ is a $C$-type intermediate ring as noted above and it follows from the Theorem \ref{2} that $C_{\mathcal{P}}(X)+C^*(X)$ is also a $C$-type intermediate ring. Since each $C$-type intermediate ring happens to be the largest member of its own equivalence class as realised in Theorem \ref{1}, this implies that $C_T=C_{\mathcal{P}}(X)+C^*(X)$ a contradiction.  
\end{proof}

\begin{remark}
	The above proof clearly suggests that a space $X$ is pseudocompact if and only if each $C$-type ring lying between $C^*(X)$ and $C(X)$ is of the form $C^*(X)+I$ for some ideal $I$ of $C(X)$.
\end{remark}
 
\section{When does the class of locally-$\mathcal{P}$ spaces$/$almost locally-$\mathcal{P}$ spaces becomes closed under product?} 
 The following result gives a description of the members of products of ideals of closed sets. The proof of this fact is routine.
 
\begin{theorem}\label{4}
Let $X$ be the product of an indexed family $\{X_{\alpha}:\alpha\in\Lambda\}$ of topological spaces and $\mathcal{P}\equiv\prod_{\alpha\in\Lambda}\mathcal{P}_{\alpha}$, where for each $\alpha\in\Lambda,~\mathcal{P}_{\alpha}\in\Omega(X_\alpha)$. Then $\mathcal{P}=\{A\subseteq X: A\text{ is closed in }X \text{ and }A\subseteq \prod_{\alpha\in\Lambda}A_\alpha$, where $A_\alpha\in\mathcal{P}_\alpha\text{ for each }\alpha\in\Lambda\}$.
\end{theorem} 

With the notation of Theorem \ref{4}, the next proposition tells us precisely when the product of locally-$\mathcal{P}_\alpha$-spaces becomes locally-$\mathcal{P}$.

\begin{theorem}\label{5}
$X=\prod_{\alpha\in\Lambda}X_\alpha$ is a locally-$\mathcal{P}$space if and only if 
\begin{enumerate}
\item $X_\alpha$ is locally-$\mathcal{P}_\alpha$ for each $\alpha\in\Lambda$.
\item $X_\alpha\in\mathcal{P}_\alpha$ for all but possibly finitely many $\alpha$'s in $\Lambda$. 
\end{enumerate}
\end{theorem}

\begin{proof}
Let $X$ be locally-$\mathcal{P}$, choose a point $x_\beta$ from the space $X_\beta, \beta\in\Lambda$. Let $x$ be a point in $X$, whose $\beta$ th co-ordinate is $x_\beta$. On using Theorem \ref{4}, we can find out an open neighbourhood $U$ of $x$ in $X$ such that $\cl_X U\subseteq \prod_{\alpha\in\Lambda}A_\alpha$, where $A_\alpha\in\mathcal{P}_\alpha$ for each $\alpha\in\Lambda$. We can write $U=\prod_{\alpha\in\Lambda}U_\alpha$, where each $U_\alpha$ is open in $X_\alpha$ and $U_\alpha\neq X_\alpha$ for at most finitely many $\alpha$'s in $\Lambda$. It follows that $\prod_{\alpha\in\Lambda}\cl_{X_\alpha}U_\alpha\subseteq \prod_{\alpha\in\Lambda}A_\alpha$ and hence $\cl_{X_\beta}U_\beta\subseteq A_\beta$ implying that $\cl_{X_\beta}U_\beta\in\mathcal{P}_\beta$. Thus $X_\beta$ becomes locally-$\mathcal{P}_\beta$ and it is easy to see that $A_\alpha=X_\alpha$ for all but finitely many $\alpha$'s and consequently $X_\alpha\in\mathcal{P}_\alpha$ for all these $\alpha$'s.

Conversely, let $X_\alpha$ be locally-$\mathcal{P}_\alpha$ for each $\alpha\in\Lambda$ and let there be a finite subset $\Lambda_\circ=\{\alpha_1,\alpha_2,...,\alpha_n\}$ of $\Lambda$ for which $X_\alpha\in\mathcal{P}_\alpha$ for all $\alpha\in\Lambda\setminus\Lambda_\circ$. Choose a point $x$ from $X$, with its co-ordinate $x_\beta$ for $\beta\in\Lambda$. For each $\alpha\in\Lambda\setminus\Lambda_\circ$, we select an open neighbourhood $V_\alpha$ of $x_\alpha$ in $X_\alpha$ such that $\cl_{X_\alpha}V_\alpha\in\mathcal{P}_\alpha$. Let $V=\prod_{\alpha\in\Lambda}V_\alpha$, with $V_\alpha=X_\alpha$ whenever $\alpha\in\Lambda\setminus\Lambda_\circ$. Then $V$ is an open neighbourhood of $x$ in $X$ and $\cl_XV=\prod_{\alpha\in\Lambda}\cl_{X_\alpha}V_\alpha\in\mathcal{P}$. Thus $X$ becomes locally-$\mathcal{P}$.
\end{proof}

The following fact follows by adopting similar arguments:

\begin{theorem}\label{6}
$X=\prod_{\alpha\in\Lambda}X_\alpha$ is an almost locally-$\mathcal{P}$-space if and only if 
\begin{enumerate}
\item $X_\alpha$ is almost locally-$\mathcal{P}_\alpha$ for each $\alpha\in\Lambda$ and
\item $X_\alpha\in\mathcal{P}_\alpha$ for all but finitely many $\alpha$'s in $\Lambda$.
\end{enumerate}
\end{theorem}

We shall now record a particularly interesting special case of Theorem \ref{5} and \ref{6} For this purpose we need to recall that a subset $A$ of $X$ is called bounded if each $f\in C(X)$ is bounded on $A$. Every pseudocompact subset of $X$ is bounded. The following theorem of Mandelkar, proved in 1971 inform us that the converse of the last statement is true for a special class of subsets of $X$.

\begin{theorem}(\cite{M})\label{7}
A support in $X$ i.e. a subset of the form $\cl_X(X-Z(f)), f\in C(X)$ is bounded if and only if it is pseudocompact.
\end{theorem}

We call a space $X$ locally pseudocompact( almost locally pseudocompact) if each point $x$ on it has an open neighbourhood $U_x$  whose closure is pseudocompact(respectively each nonempty open set $U$ in $X$ contains a nonempty open set $V$ such that $\cl_XV$ is pseudocompact).

Let us call a space $X$ locally bounded ( almost locally bounded ) if $X$ is locally-$\mathcal{B}$ ( respectively almost locally-$\mathcal{B}$ ), where $\mathcal{B}$ is the ideal of all closed bounded sets in $X$. Since each neighbourhood of a point in $X$ contains a cozero set neighbourhood of the same point, the following proposition is an easy consequence of Theorem \ref{7}. 
 
\begin{theorem}\label{8}
$X$ is locally pseudocompact/ almost locally pseudocompact if and only if $X$ is locally bounded ( respectively almost locally bounded).
\end{theorem}

The following propositions now emerge as special cases of Theorem \ref{5} and  \ref{6} respectively.

\begin{theorem}\label{9}
The product space $X\equiv \prod_{\alpha\in\Lambda}X_\alpha$ is locally pseudocompact if and only if each $X_\alpha$ is locally pseudocompact and for all but finitely many $\alpha$'s, $X_\alpha$'s are pseudocompact.
\end{theorem}

\begin{theorem}
The product space $X\equiv \prod_{\alpha\in\Lambda}X_\alpha$ is almost locally pseudocompact if and only if each $X_\alpha$ is almost locally pseudocompact and for all but finitely many $\alpha$'s, $X_\alpha$'s are pseudocompact.
\end{theorem}

We would like to mention at this point that, since the closed pseudocompact subspaces of a space $X$ may not make an ideal of closed sets, indeed a closed subset of a pseudocompact space need not be pseudocompact (vide: the right edge $\{\omega_1\}\times \omega_\circ$ of the Tychonoff plank $T\equiv [0,\omega_1]\times [0,\omega_\circ]\setminus\{(\omega_1,\omega_\circ)\}$ is not pseudocompact), we cannot deduce the Theorem \ref{8} and \ref{9} directly from the general propositions \ref{5} and \ref{6} respectively. Incidentally the following example of an almost locally compact space which is not locally compact is constructed in \cite{A}.

\begin{example}
 Let $X$ be an uncountable set, each point of which excepting a point $p\in X$ is isolated and a typical neighbourhood of $p$ is a set $A$ containing $p$ such that $X\setminus A$ is at most countable set.
\end{example}

The following example is quite analogous to the above one.

\begin{example}	
  Let $X$ be an uncountable set with $\lvert X\rvert >\omega_1$ with the following topology: each point of $X$, barring a distinguished point $p$ is isolated and a typical neighbourhood of $p$ is a set $A$ containing $p$ for which $\lvert X\setminus A\rvert \leq \omega_1$. It is easy to check that if $\mathcal{P}$ is any one of the three ideals of closed sets viz $\mathfrak{F}, \mathfrak{B}$ and $\mathfrak{L}$, then $X$ is almost locally-$\mathcal{P}$ without being locally-$\mathcal{P}$. Here $\mathfrak{F}$ stands for the ideal of all finite subsets of $X$ and $\mathfrak{L}$ that for all closed Lindelof subsets of $X$.
\end{example}

\section{On the question of freeness and essentiality of ideals ideals $C_\mathcal{P}(X)$ and $C^{\mathcal{P}}_{\infty}(X)$} 
The following proposition relates the local $\mathcal{P}$-condition on $X$ with the freeness of the ideals $C_\mathcal{P}(X)$ and $C^{\mathcal{P}}_{\infty}(X)$.

\begin{theorem}
For $\mathcal{P}\in\Omega(X)$, the following three statements are equivalent:
\begin{enumerate}
\item $X$ is locally-$\mathcal{P}$.
\item $C_\mathcal{P}(X)$ is a free ideal of $C(X)$.
\item $C^{\mathcal{P}}_\infty(X)$ is a free ideal of $C^{\mathcal{P}}_{\infty}(X)+C^*(X)$.
\end{enumerate}
\end{theorem}

\begin{proof}
$(1)\Rightarrow (2):$ Let $X$ be locally-$\mathcal{P}$ and $x\in X$. Then there exists an open neighbourhood $U$ of $x$ such that $\cl_XU\in\mathcal{P}$. By using the complete regularity of $X$, we can find an $f\in C^*(X)$ such that $f(X\setminus U)=0$ and $f(x)=1$. This implies that $X\setminus Z(f)\subseteq U$ and consequently $\cl_X(X\setminus Z(f))\subseteq \cl_X U$ and hence $\cl_X(X\setminus Z(f))\in\mathcal{P}$ i.e. $f\in C_{\mathcal{P}}(X)$, we note that $f(x)\neq 0$.

$(2)\Rightarrow (3):$ is trivial as $C_\mathcal{P}(X)\subseteq C^{\mathcal{P}}_{\infty} (X)$.

$(3)\Rightarrow (1):$ Let $(3)$ hold and $x\in X$. Then there exists $f\in C^{\mathcal{P}}_{\infty}(X)$ such that $f(x)\neq 0$, say $\lvert f(x)\rvert =\lambda$. Take $U=\{y\in X:\lvert f(y)\rvert >\frac{\lambda}{2}\}$, which is a nonempty open set. Since $\cl_XU\subseteq \{ y\in X:\lvert f(y)\vert\geq \frac{\lambda}{2}\}$ and the last set is a member of $\mathcal{P}$, it follows that $\cl_XU\in\mathcal{P}$. Thus $X$ becomes locally-$\mathcal{P}$.
\end{proof}

We would like to mention that a variant condition using a different terminology both necessary and sufficient for the freeness of the ideal $C_\mathcal{P}(X)$ in $C(X)$ have also been discovered in [\cite{T}, proposition 2.11]. 

\begin{remark}
With $\mathcal{P}\equiv \mathcal{K}\equiv $ the ideal of compact sets in $X$, the above theorem reads: $X$ is locally compact if and only if $C_K(X)$ is a free ideal of $C(X)$ if and only if $C_\infty (X)$ is a free ideal of $C^*(X)$, a standard result in Rings of Continuous Functions [see \cite{G}].
\end{remark}

\begin{remark}\label{10}
 On choosing $\mathcal{P}\equiv\mathcal{B}$, it follows from Mandelkar's theorem 3.4 that $C_\psi(X)=C_\mathcal{B}(X)$ and $C^\psi_\infty(X)=C^\mathcal{B}_\infty(X)$. Thus with this particular choice of $\mathcal{P}$, Theorem 4.1 states that $X$ is locally pseudocompact when and only when $C_\psi(X)$ is a free ideal of $C(X)$, equivalently $C^\psi_\infty(X)$ is a free ideal of $C^*(X)$.
\end{remark} 
 Before embarking on the question of essentiality of the ideals $C_\mathcal{P}(X)$ and $C^\mathcal{P}_\infty(X)$, we write down the following topological description of essential ideals of an intermediate ring. The proof is accomplished by following the technique of \cite{A}.

\begin{theorem}\label{11}
	A nonzero ideal $I$ in an intermediate ring $A(X)$ is essential in this ring if and only if $\cap Z[I]\equiv \cap_{f\in I}Z(f)$ is a nowhere dense subset of $X$.
\end{theorem}

It follows from this Theorem that every free ideal of $A(X)$ is an essential ideal. Also for any $\mathcal{P}\in\Omega (X)$, the local $\mathcal{P}$ condition on $X$ is stronger than almost local $\mathcal{P}$. The following proposition which correlates almost local $\mathcal{P}$ requirement with essentiality of the ideals $C_{\mathcal{P}}(X)$ and $C^\mathcal{P}_\infty(X)$ seems therefore natural.

See also the proposition 3.2 in \cite{T} in this context.
\begin{theorem}\label{12}
For any $\mathcal{P}\in\Omega(X)$, the following statements are equivalent:
\begin{enumerate}
\item $X$ is almost locally-$\mathcal{P}$.
\item $C_{\mathcal{P}}(X)$ is an essential ideal of $C(X)$.
\item $C^\mathcal{P}_\infty(X)$ is an essential ideal of $C^*(X)+C^\mathcal{P}_\infty(X)$.
\end{enumerate}
\end{theorem}  

\begin{proof}
$(1)\Rightarrow(2):$ Suppose that $X$ is almost locally-$\mathcal{P}$. Let $f\neq 0$ be a non unit function in $C(X)$. We shall show that there is a $g\in C_\mathcal{P}(X)$ such that $fg\neq 0$. By using the regularity of $X$ and the local $\mathcal{P}$ condition on it, we can find nonempty open sets $A$ and $B$ in $X$ such that $\cl_XB\subseteq A\subseteq \cl_XA\subseteq X\setminus Z(f)$ and $\cl_XB\in\mathcal{P}$. We now use the complete regularity of $X$ to produce a $g\in C^*(X)$ such that $g(X\setminus B)=0$ and $g(b)=1$ for some point $b\in B$. Hence $X\setminus Z(g)\subseteq B$. Consequently $\cl_X(X\setminus Z(g))\subseteq \cl_X B$. Since $\cl_XB\in\mathcal{P}$, it follows that $\cl_X(X\setminus Z(g))\mathcal{P}$ i.e. $g\in C_\mathcal{P}(X)$. We observe that $f(b)g(b)\neq 0$.

$(2)\Rightarrow (3):$ This is trivial because $C_\mathcal{P}(X)\subseteq C^\mathcal{P}_\infty(X)$.

$(3)\Rightarrow (1):$ The proof is accomplished on using Theorem \ref{11}, taking care of the fact that if $f\in C^\mathcal{P}_\infty(X)$ and $g\in C^*(X)$, then $fg\in C^\mathcal{P}_\infty(X)$ and following the arguments in \cite{AZ}, closely.
\end{proof}

The following proposition is a consequence of Theorem \ref{8}, first two lines in Remark \ref{10} and Theorem \ref{12}.

\begin{theorem}
The statements written below are equivalent.
\begin{enumerate}
\item $X$ is almost locally pseudocompact.
\item $C_\psi(X)$ is an essential ideal of $C(X)$.
\item $C^\psi_\infty(X)$ is an essential ideal of $C^*(X)$.
\end{enumerate}
\end{theorem}

The authors in \cite{KR}, have observed that for a large class of spaces $X$, the socle $S$ of $C(X)$ turns out to be the same as the ring $C_K(X)$ and have asked for a characterisation of space $X$ with such property. Motivated by this question, the authors in \cite{A} offered a simple topological characterization of such spaces and a special subclass of those spaces in the following manner.

\begin{theorem}\label{13}(Theorem 4.5\cite{A})
\begin{enumerate}
\item $C_K(X)=S$ if and only if $X$ is pseudodiscrete meaning that each compact subset of $X$ has finite interior.
\item $C_\infty(X)=S$ if and only if $X$ is pseudodiscrete and contains at most a finite numbers of isolated points.
\end{enumerate}
\end{theorem}

Being prompted by this result, we have established the following fact which places this result in a more general setting. Given $\mathcal{P}\in\Omega(X)$, we call $X$, $\mathcal{P}$-pseudodiscrete if each member of $\mathcal{P}$ has finite interior.

\begin{theorem}\label{14}
Let $\mathcal{P}\in\Omega(X)$. Suppose that either $X$ has no isolated point or each singleton subset of $X$ is a member of $\mathcal{P}$. Then 
\begin{enumerate}
\item $C_\mathcal{P}(X)=S$ if and only if $X$ is $\mathcal{P}$ pseudodiscrete.
\item $C^\mathcal{P}_\infty(X)=S$ if and only if $X$ is $\mathcal{P}$ pseudodiscrete and contains at most a finite numbers of isolated points.
\end{enumerate}
\end{theorem}

\begin{proof}
\begin{enumerate}
\item Since the socle $S$ of $C(X)$ contains precisely those functions in $C(X)$, which vanish everywhere except on a finite number of points of $X$, as is established in \cite{KR}, Proposition 3.3, it follows on using the hypothesis at the beginning of the statement of this theorem that $S\subseteq C_\mathcal{P}(X)$. The reverse containment $C_\mathcal{P}\subseteq S$ can be proved by using the $\mathcal{P}$-pseudodiscreteness of $X$ and by adapting carefully the arguments in the proof of Theorem 4.7 (1).
\item This can be proved by closely following the arguments of proof of Theorem \ref{13} (2) and making certain obvious modifications.
\end{enumerate}
\end{proof}

We would like to record a special case of Theorem 4.8, on using $\mathcal{P}\equiv \mathcal{B}$. To do so, we need the following fact.

\begin{theorem}\label{15}
The pair of statements written below are equivalent:
\begin{enumerate}
\item Every closed pseudocompact subset of $X$ has finite interior.
\item Every closed bounded subset of $X$ has finite interior.
\end{enumerate}
\end{theorem}

\begin{proof}
$(2)\Rightarrow (1):$ This is trivial. Assume that $(1)$ holds. Take a closed bounded set $A$ in $X$. If possible let $int_XA$ be infinite. We choose a countable infinite subset $B=\{b_1,b_2,...,b_n,....\}$ of $int_XA$. For each $i\in \mathbb{N}$, there exists an $f_i\in C(X)$ such that $b_i\in X\setminus Z(f_i)\subseteq int_XA$. This implies that $B\subseteq \cup_{i=1}^{\infty}(X\setminus Z(f_i))\subseteq int_XA$, hence $B\subseteq X\setminus \cap_{i=1}^\infty Z(f_i)=X\setminus Z(f)\subseteq int_XA\subseteq A$ for some $f\in C(X)$. This implies that $B\subseteq X\setminus Z(f)\subseteq \cl_X (X\setminus Z(f))\subseteq A$. Now since $A$ is bounded so also is $\cl_X(X\setminus Z(f))$, which in view of Theorem \ref{7}, turns out to be a closed pseudocompact space with infinite interior. A contradiction is arrived at. 
\end{proof}

\begin{theorem}
For a space $X$,
\begin{enumerate}
\item $C_\psi(X)=S$ if and only if every closed pseudocompact subset of $S$ has finite interior.
\item $C^\psi_\infty(X)=S$ if and only if every closed pseudocompact subset of $X$ has finite interior and $X$ contains not more than finitely many isolated points.
\end{enumerate}
\end{theorem}

\begin{proof}
Follows from Theorem \ref{14} and Theorem \ref{15}.
\end{proof}

\begin{example}
	The conclusions of Theorem \ref{14} may not hold good if the hypothesis that each singleton subset of $X$ is a member of $\mathcal{P}$(in case $X$ contains at least one isolated point) is dropped.
	Indeed we take any infinite Tychonoff space $X$, where there is just one isolated point say $p$. Let $\mathcal{P}$ be any member of $\Omega(X)$ with the following two conditions: 

\begin{enumerate}
\item If $A\in \mathcal{P}$ then $p\notin A$.
\item Interior of each member of $\mathcal{P}$ is finite.	
\end{enumerate}	

To be specific we can choose $\mathcal{P}$ to be the family of all finite subsets of $X$ which miss the point $p$ or the family of all closed nowhere dense subsets of $X$, missing the point $p$. The function $f\in C(X)$ given by $f(p)=1$ and $f(q)=0$ for all $p\neq q$ in $X$, belongs to the socle $S$ of $C(X)$ in view of Proposition 3.3 in \cite{KR}. But we see that $\cl_X(X\setminus Z(f))=\{p\}$, which does not belong to $\mathcal{P}$, which means that $f\notin C_\mathcal{P}(X)$. Altogether $S\neq C_\mathcal{P}(X)$. It is also easy to check that $f\notin C^\mathcal{P}_\infty(X)$ because $\{x\in X:\lvert f(x)\rvert \geq \frac{1}{2}\}=\{p\}$ does not belong to $\mathcal{P}$. Then $S\neq C^\mathcal{P}_\infty(X)$.	
\end{example}

\section{Ideals $C_\mathcal{P}(X)$ are the same as $z_\circ$-ideals in $C(X)$}
For an element $a$ of a commutative ring $R$ with identity, let $P_a$ stand for the intersection of all minimal prime ideals of $R$ which contain $a$. An ideal $I$ of $R$ is called a $z_\circ$-ideal if for each member $a$ of $I, P_a\subseteq I$.

The following fact is proved in \cite{AKA}.

\begin{theorem}\label{16}
For $f\in C(X), P_f=\{ g\in C(X):int_XZ(f)\subseteq int_XZ(g)\}$.
\end{theorem}

We now prove the main result of this section.

\begin{theorem}\label{17}
A proper ideal $I$ of $C(X)$ is a $z^\circ$-ideal if and only if there exists $\mathcal{P}\in \Omega (X)$ such that $I= C_\mathcal{P}(X)$.
\end{theorem}

\begin{proof}
Let $I$ be a $z^\circ$-ideal of $C(X)$. Suppose $\mathcal{P}[I]$ is the smallest member of $\Omega (X)$ containing the family of sets $\{\cl_X(X\setminus Z(f)):f\in I\}$. Since the family is already closed under finite intersection, as can be easily verified, it follows that $\mathcal{P}[I]=\{Y\subseteq X: Y \text{ is closed in } X \text{ and } Y\subseteq \cl_X(X\setminus Z(f))\text{ for some } f\in I\}$. We shall show that $I=C_{\mathcal{P}[I]}(X)$. It is plain that $I\subseteq C_{\mathcal{P}[I]}(X)$. To prove the other containment, choose $f$ from $C_{\mathcal{P}[I]}(X)$. Then there exists a $g\in I$ such that $\cl_X(X\setminus Z(f))\subseteq \cl_X(X\setminus Z(g))$ which clearly implies that $int_XZ(g)\subseteq int_XZ(f)$. It follows from the Theorem \ref{16} that $f\in P_g$. Also the fact that $I$ is a $z^\circ$-ideal of $C(X)$ together with $g\in I$ imply that $P_g\subseteq I$. Hence $f\in I$. Thus $C_{\mathcal{P}[I]}(X)\subseteq I$.

To prove the other part of this theorem, let $\mathcal{P}\in\Omega (X)$. We shall show that $C_\mathcal{P}(X)$ is a $z^\circ$-ideal of $C(X)$. Choose $f\in C_\mathcal{P}(X)$ and $g\in P_f$. Then from Theorem \ref{16}, we can write $int_X Z(f)\subseteq int_X Z(g)$. This implies that $\cl_X (X\setminus Z(g))\subseteq \cl_X (X\setminus Z(f))$. But $f\in C_\mathcal{P}(X)$ means that $\cl_X(X\setminus Z(f))\in \mathcal{P}$. Since $\mathcal{P}$ is an ideal of closed sets in $X$, it follows that $\cl_X(X\setminus Z(g))\in\mathcal{P}$ and hence $g\in C_\mathcal{P}(X)$. Thus we get $P_f\subseteq C_\mathcal{P}(X)$ for an arbitrary $f$ in $C_\mathcal{P}(X)$. This settles that $C_\mathcal{P}(X)$ is a $z^\circ$-ideal of $C(X)$.
\end{proof}

\begin{corollary}
Let an ideal $I$ in $C(X)$ contain divisors of zero only. Then $I$ is contained in a proper $z^\circ$-ideal of $C(X)$.
\end{corollary}

\begin{proof}
It follows from the proof of  Theorem \ref{17}, first part that $I\subseteq C_{\mathcal{P}[I]}(X)$. Since each $f\neq 0$ in $I$ is a divisor of zero, it follows that $int_XZ(f)$ is nonempty proper open subset of $X$. Consequently for any such $f, \cl_X(X\setminus Z(f))$ is a nonempty proper closed subset of $X$ and therefore $X\notin \mathcal{P}[I]$. It follows that $1\notin C_{\mathcal{P}[I]}(X)$.
\end{proof}
\bibliographystyle{plain}

\end{document}